\documentclass[a4paper,11pt]{amsart}
\usepackage{amssymb,amsfonts,amsxtra,
	mathrsfs,graphicx,verbatim,stmaryrd,hyperref,cite,color, tikz-cd}
\usepackage[margin=9mm,
marginparwidth=20mm,     % + <- Width of your marginpar
marginparsep=1mm,       % + <- Gap between text block and marginpar
]
{geometry}
\usepackage[all]{xy}
\xyoption{line}
\usepackage{fullpage}
\usepackage{euscript}
\newtheorem{theorem}{Theorem}[section]
\newtheorem{cor}[theorem]{Corollary}
\newtheorem{lem}[theorem]{Lemma}
\newtheorem{prop}[theorem]{Proposition}

\theoremstyle{definition}

\newtheorem{defi}[theorem]{Definition}
\newtheorem{rem}[theorem]{Remark}

\numberwithin{equation}{section}
\DeclareMathOperator{\B}{\mathcal B}
\DeclareMathOperator{\C}{\mathcal C}
\DeclareMathOperator{\D}{\mathcal D}
\DeclareMathOperator{\dgcat}{dgCat}
\DeclareMathOperator{\DGA}{DGA}
\DeclareMathOperator{\Hom}{Hom}
\DeclareMathOperator{\End}{End}

\newcommand{\noproof}{\begin{flushright} \ensuremath{\square}
\end{flushright}}
\def\ground{\mathbf{k}}

\def\id{\operatorname{id}}

\thanks{}

\begin{document}
\begin{abstract}
	The notion of one-sided localization in the homotopy invariant context is developed for dg algebras and dg categories. Applications include a simple construction of derived localization of dg algebras and dg categories, and a refinement of Drinfeld's quotient of pretriangulated dg categories.
\end{abstract}
\title[One-sided localization in dg categories]{One-sided localization in dg categories}
\author{Joe Chuang, Andrey Lazarev}
\thanks{This work was partially supported by EPSRC grants EP/T029455/1 and EP/N016505/1}

\thanks{}

\maketitle
\tableofcontents
\section{Introduction}
Localization is a standard and important tool in algebra; its early applications allow one to construct integers out of natural numbers and rational numbers out of integers. In the noncommutative context localization exhibits considerable subtleties and is best treated as a nonabelian derived functor in the framework of model categories, cf. \cite{Toen, BCL}.

There is also a natural notion of \emph{one-sided} localization. Given an associative algebra $C$ over a commutative ring $\ground$ and an element $v\in C$, one can form the algebra $C_v:=C\langle w\rangle/(vw=1)$. The algebra $C_v$ comes equipped with a map from $C$ such that the image of $v$ in $R_v(C)$ has a right inverse. It clearly has an appropriate universal property characterizing it. This construction (in a more general context of one-sided matrix inversion) was studied in \cite{Cohn}.

In the present work we consider the problem of constructing one-sided localization in a homotopy invariant context. Thus, we start with a differential graded (dg) associative algebra $C$ and a cycle $v\in A$. Then the derived right inversion (or localization) of $A$ at $v$ is given by $R_v(C):=A\langle w,u\rangle$ with the differential $d(u)=vw-1$; the left derived inversion is defined symmetrically. Still more generally, $C$ can be taken to be a dg category and $v$ to be a closed morphism in it; then one can form a dg category $R_v(C)$ where $v$ is `universally inverted on the right up to homotopy'. The latter expression needs explaining, and much of the paper is about making sense of it and putting it in a proper context. One application of the developed theory is a simple construction of the derived (two-sided) localization of a dg category as a composition of a left- and right-sided localizations, and its characterization by a universal property, cf. Proposition \ref{prop:2sideduniversal} below. This is a version of a result originally due to To\"en, \cite[Corollary 8.7]{Toen}.

We also consider a simpler construction of \emph{killing} a cycle in a dg algebra (or, more generally, a closed morphism in a dg category). Given a dg algebra $C$ and a cycle $v\in C$ we can form $C/v:=C\langle w\rangle$ with $d(w)=v$. The construction $C/v$ does not depend on the choice of $v$ in its homology class and enjoys an appropriate universal property up to homotopy. 

Our main result concerns the situation when $\C$ is a pretriangulated dg category (such as the category of complexes of modules over a ring). Given a  morphism $v:B\to C$ in $\C$, we extend it to a triangle 
\begin{equation}\label{eq:triangle}	\begin{tikzcd}
		A \arrow[r,"x"] &
		B \arrow[r,"v"] &
		C \arrow[r,"s"] & \Sigma A
	\end{tikzcd}
\end{equation}
Then the operation of killing $v$ in $\C$ is equivalent to left-localizing at $s$ or right-localizing at $x$. This is the content of Theorem \ref{killingvslocalisation}. If $B=C$ and the map $v$ is the identity map, then killing $v$ is precisely the dg localization of Drinfeld's, cf. \cite{Drinfeld} which is the dg-version of the Verdier quotient by the object $B=C$. As an application, we obtain the following refinement of Drinfeld's construction.  Given a triangle (\ref{eq:triangle}), the Drinfeld quotient of $\C$ by the full pretriangulated subcategory generated by $B$ is quasi-equivalent to either $(\C/x)/v$ or $(\C/v)/x$. In particular, the dg categories $(\C/x)/v$ or $(\C/v)/x$ are quasi-equivalent and do not depend, up to quasi-equivalence, on  individual morphisms $v$ and $x$ so long as they are consecutive morphisms of a triangle in $\C$.

The paper is organized as follows. In Section 2 we construct a relative cylinder in the category of dg categories; this is a categorical version of the Baues-Lemaire cylinder for dg algebras \cite{BL}. This construction is simpler than that of an absolute cylinder and is likely to be useful in other contexts. Section 3 introduces the notion of the killing of a closed morphism in a dg category and gives its homotopy invariant characterization using the previously developed construction of a cylinder. Section 4 similarly introduces and gives homotopy characterization of a one-sided localization whereas in Section 5  one-sided localization is used to construct ordinary, i.e. two-sided localization. The case of a dg algebra (which can be viewed as a dg category with one object) is considered in parallel and similar results are obtained. Finally, in Section 6 the case of pretriangulated categories is treated and the refinement of Drinfeld's quotient, mentioned above, is constructed.
\subsection{Notation and conventions}
We work in the category of dg modules over a fixed commutative ring $\ground$; the notation $\Hom$ stands for the set of $\ground$-linear homomorphisms. All dg-modules are homologically graded and we will write $\Sigma$ for the homological suspension. A dg category is a category enriched over dg $\ground$-modules and they will be denoted by calligraphic letters such as $\C$. The category of dg categories will be denoted by $\dgcat$. The subcategory of $\dgcat$ consisting of dg categories with one object will be denoted by $\DGA$; it is the category of dg algebras. It is well-known that $\dgcat$ and DGA form  model categories \cite{Tabuada, Hinich}; we refer to \cite{Kel} for the background material and terminology on dg categories and dg functors. The model categories $\dgcat$ and $\DGA$ are not left proper unless $\ground$ is a field and this may lead to homotopy non-invariant constructions; to avoid this, we will always tacitly assume our dg categories or dg algabras to be flat over $\ground$. Given an object $O$ in a category $\C$, we will denote by $O\downarrow\C$ the category of objects under $O$, i.e. morphisms in $\C$ of the form $O\to X$ with maps between such morphisms being obvious commutative triangles.

\section{Relative cylinder in dg categories}
Let $\C,\D$ be  dg categories and $F:\C\to\D$ be a dg functor. We assume that $\C$ and $\D$ have the same collection of objects and that $F$ acts identically on objects. In addition, we assume that $\D$ is free relative to $\C$ when differentials are disregarded. In other words, $\D$ is generated, as a nondifferential category, by $\C$ and a collection of morphisms $s_\alpha\in \Hom(\D), \alpha\in I$ where $I$ is an indexing set with no relations imposed between $s_\alpha$. The functor $F$ then exhibits $\C$ as a dg subcategory of $\D$. 

We will construct explicitly a cylinder object for $\D$ in $\C\downarrow\dgcat$, the under category of $\C$ in $\dgcat$. 
We start with the notion of a derivation adapted to the categorical framework.
\begin{defi}
 Let $\C_1$ and $\C_2$ be two dg categories and $F, G:\C_1\to\C_2$ be dg functors. Then an \emph{$(F,G)$-derivation} of $\C_1$ with values in $\C_2$ is a homogeneous $\ground$-linear function $f$ associating to any object $c$ in $\C_1$ an object $f(c)$ in $\C_2$ and to any morphism $x:c\to c'$ in $\C$ a morphism $f(x):f(c)\to f(c')$ such that the Leibniz rule holds: $f(x_1 x_2)=f(x_1)G(x_2)+(-1)^{|f||x_1|}F(x_1)f(x_2)$.
\end{defi}
	Consider the dg category $\D_1\coprod_{\C}\D_2$ where $\D_1$ and $D_2$ are both isomorphic to $\D$; given a morphism $s\in \Hom(\D)$, we will write $s^1$ and $s^2$ for the corresponding morphisms in  $\D_1\coprod_{\C}\D_2$. Let $C_{\C}(\D)$ be the graded category freely generated over $\D_1\coprod_{\C}\D_2$ by morphisms $\bar{s}_\alpha: d\to d'$, one for each generator $s_\alpha\in \D: d\to d'$, and such that $|\bar{s}_\alpha|=|s_\alpha|+1$.  Let $f$ be the $(i_1,i_2)$-derivation of $\D$ with values in $C_{\C}(\D)$ (where $i_1$ and $i_2$ are the two inclusions of $\D$ into $C_{\C}(\D)$) determined by the rule $f(s_\alpha)=\bar{s}_\alpha$.
	\begin{defi}	
Let the differential on  $C_{\C}(\D)$ be determined by the requirement that both $\D_1$ and $\D_2$  are dg subcategories in $C_{\C}(\D)$ and $d(\bar{s}_\alpha)=s^1_\alpha-s^2_\alpha-f(ds_\alpha)$. The dg category $C_{\C}(\D)$ so defined, is called the $\C$-relative cylinder of $\D$
\end{defi}
\begin{lem}
	The differential in $C_{\C}(\D)$ squares to zero.
\end{lem}
\begin{proof}
Let us define the $(i_1,i_2)$-derivation $[d,f]$ of $\D$ with values in $C_{\C}(\D)$ by the formula
$
[d,f](x)=d(f(x))+f(dx)$
where $x\in\Hom(\D)$. Then the following formula holds for any $x\in\Hom(\D)$:
\begin{equation}\label{eq:derivation}
[d,f](x)=x^1-x^2
\end{equation}
Indeed, both sides of the above are $(i_1,i_2)$-derivations that agree when applied to any generator $s_\alpha$ using the definition of $f$; this confirms the validity of (\ref{eq:derivation}) in general. Next, we have for any generator $s_\alpha\in\Hom(\D)$:
\begin{align*}
d^2(\bar{s}_\alpha)&=d(s_\alpha^1-s_\alpha^2-f(ds_{\alpha}))\\
&=ds_\alpha^1-ds_\alpha^2-d(f(ds_{\alpha}))\\&=ds_\alpha^1-ds_\alpha^2-[d,f](ds_\alpha)\\&=0
\end{align*}
where (\ref{eq:derivation}) is used for the last equality.
\end{proof}	
Note that the map $C_{\C}(\D)\to \D$ sending both $\D_1,\D_2$ identically to $\D$ and the generators $\bar{s}$ to zero is a map of dg categories. By construction, $\D\coprod_{\C}\D$ is a dg subcategory of $C_{\C}(\D)$. In other words, $C_{\C}(\D)$ factors the codiagonal map $\D\coprod_{\C}\D\to \D$ in  $\C\downarrow\dgcat$.

We now assume, in addition, that $\D$ is \emph{cofibrant} over $\C$. More specifically, we assume that there is an additional weight indexing by natural numbers on the set of generators $s_\alpha$ so that $d(s_{\alpha})$ is of a sum of elements of strictly smaller weight, where the weight of $\C$ is taken to be zero.

% $\D$ is generated over $\C$ by the set $s_{\alpha,i},i=1,2,\ldots$ and the differential $d(s_{\alpha,i})$ belongs to the span of compositions of $s_{\alpha, j}$ with $j<i$. Under this assumption, we have the following result.
\begin{prop}
	The dg category $C_{\C}(\D)$ is a good cylinder object in $\C\downarrow\dgcat$.
\end{prop}
\begin{proof}
One only needs to prove that the map $C_{\C}(\D)\to \D$ described above is a quasi-equivalence. Let us consider first the case when the differential in $\D$ (and so also in $\C$) is zero. Let $\C\langle \{s_\alpha\}\rangle$ be the graded quiver whose vertices are the objects of $\C$ and graded spaces of arrows are spanned by $\Hom(\C)$ together with the generators $s_\alpha, \alpha\in I$. Let $C_{\C}(\C\langle \{s_\alpha\}\rangle)$ be the dg quiver that has the same vertices as $\C\langle \{s_\alpha\}\rangle$ and three types of basis arrows: $s^1,s^2$ and $\bar{s}$ for every generator $s$; the differential has the form $d(\bar{s})=s^1-s^2$. It is clear that the homology of the latter quiver is isomorphic to $\C\langle \{s_\alpha\}\rangle$. The dg category $C_{\C}(\D)$ is isomorphic to the $\C$-relative free category on  $C_{\C}(\C\langle \{s_\alpha\}\rangle)$. Since the formation of the free category commutes with the homology, the desired statement follows.

Now let us consider the general case. Since the set of objects is constant throughout, it suffices to show that for any two objects $d,d'$, the map 
$\Hom_{C_{\C}(\D)}(d,d')\to \Hom_{\D}(d,d')$ 
is a quasi-isomorphism. Consider the increasing filtration on $\Hom_{C_{\C}(\D)}(d,d')$
\[F_0:=\Hom_{\D_1\coprod_{\C}\D_2}(d,d')\subset F_1\subset \ldots F_p\subset\ldots. \] 
Here the filtration component $F_p$ is a dg subspace of $\Hom_{C_{\C}(\D)}(d,d')$ consisting of morphisms that are linear combinations of products of generators
of weights summing to at most $p$, where the weights of $s^1_{\alpha}, s^2_{\alpha}$ and $\bar{s}_{\alpha}$ are the same as that of $s_\alpha$.
 %compositions of morphisms of the form $s^1_{\alpha, p}, s^2_{\alpha, p}$ and $\bar{s}_{\alpha, p}$. in $\Hom_{\D_1\coprod_{\C}\D_2}(d,d')$ and no more than $p$ generators $\bar{s}_\alpha$.
  It is clear this this filtration is compatible with the differential and exhaustive. The differential $d_1$ of the associated spectral sequence is induced by the differentials applied to the generators $\bar{s}_{\alpha}$  so, the $E_1$-page is isomorphic to $\Hom_{C_{\C}(\operatorname{gr}\D)}(d,d')$ where $\operatorname{gr}\D(d,d')$ is the associated graded with respect to the filtration considered above. Comparing it to the spectral sequence of $\Hom_{C_{\C}(\D)}(d,d')$, we conclude that they are isomorphic and so $C_{\C}(\D)\to \D$ is indeed a quasi-equivalence.
\end{proof}

\begin{rem}
	This result is a categorical version of the Baues-Lemaire cylinder for free differential graded algebras, see \cite{BL}; in the case of categories with one object and $\C=\ground$, it specializes to the Baues-Lemaire result. However, in the situation of dg algebras, working relative $\ground$ is equivalent to working absolutely whereas in the categorical framework it is essential to work with categories with a fixed set of objects (which is achieved by considering relative categories of the specified kind), otherwise a cylinder object will be more complicated. 
\end{rem}
\section{Killing morphisms in a dg category}
In this section we describe the procedure of killing a morphism in a dg category in a homotopy invariant fashion.
Let $\C$ be a dg category, $X,Y$ be objects of $\C$ and $v\in\Hom_n(X,Y)$ be a closed morphism between $X$ and $Y$, so $d(v)=0$. We will construct a functor
\[ \kappa_v:\C\downarrow\dgcat\to\operatorname{Sets}\]
from the undercategory of $\C$ to the category of sets.
\begin{defi} Given a dg category $\D$ supplied with a dg functor $F:\C\to\D$, set
\[
\kappa_v(\D)=\{w\in \Hom_{\D}(F(X),F(Y))_{n+1}:d(w)=F(v)\}/B_{n+1}[\Hom_{\D}(F(X),F(Y))].
\] 
\end{defi}
\begin{rem}
Note that if $w$ is such that $d(w)=v$ and for $u\in \Hom_{\D}(F(X),F(Y))_{n+1}$ it holds that $d(u)=0$, then $d(w+u)=w$, so the quotient by boundaries $B_{n+1}[\Hom_{\D}(F(X),F(Y))]$ is well-defined. Clearly the set $\kappa_v(\D)$ is an affine space over $H_{n+1}[\Hom_{\D}(F(X),F(Y))]$, in which case it is noncanonically isomorphic to $H_{n+1}[\Hom_{\D}(F(X),F(Y))]$, or it is empty. 
\end{rem}
It is easy to see that $\kappa_v$ only depends, up to a natural isomorphism, on the homology class of $v$. 

Let $\ground\langle v\rangle$ be the category with two objects and a one-dimensional space of morphisms of degree $n$ between them (here $v$ stands for a basis vector for this space). Let $\ground\langle v, w\rangle$ be the (dg) category with two objects and two-dimensional space of morphisms spanned by $v$ and $w$, in degrees $n$ and $n+1$, respectively, with the differential $d(w)=v$. There is an obvious functor $\ground\langle v\rangle\to \ground\langle v, w\rangle$ that is identical on objects and on the morphism $v$; note that $\ground\langle v, w\rangle$ is cofibrant over $\ground\langle v\rangle$. Similarly there is a functor $\ground\langle v\rangle \to \C$, sending $v$ to the specified morphism in $\C$.  This allows us to form the dg category $\C/v:=\C\coprod_{\ground\langle v\rangle}\ground\langle v,w\rangle.$ It is clear that $\C/v$ is supplied with a dg functor from $\C$ and the quasi-equivalence class of $\C/v$ as an object in $\C\downarrow\dgcat$ depends only on the homology class of $v$. We say that $\C/v$ is obtained from $\C$ by killing (the homology class of) the morphism $v$. Note that $\ground\langle v, w\rangle$ is quasi-equivalent to the coproduct $\ground\coprod \ground$ in $\dgcat$ so that $\C/v\simeq \C\coprod_{\ground[v]}^{\mathbb{L}}(\ground\coprod \ground)$; here the superscript $\mathbb L$ indicates the homotopy (derived) pushout of dg categories.

\begin{prop}\label{prop:killcat} Given a dg category $\C$, a closed morphism $v\in\Hom(\C)$ and a dg category $\D$ supplied with a dg functor $\C\to\D$ there is a natural isomorphism  of sets
\[
\kappa_v(\D)\cong [\C/v,\D]_{\C\downarrow\dgcat}
\]
where the right hand side above refers to homotopy classes of maps in $\C\downarrow\dgcat$. In other words, the functor $\kappa_v$ is represented up to homotopy by $\C/v$.
\end{prop} 
\begin{proof}
		A dg functor $\C/v\to\D$ in the undercategory of $\C$ is equivalent to choosing a morphism $w\in \Hom(D)$ for which $d(w)=F(v)$. The cylinder object $C_{\C}(\C/v)$ has the same objects as $\C$, freely generated over $\C$ by morphisms $w$, $w'$ and $s$, with $d(w)=v$, $d(w')=v$ and $d(s)=w-w'$.  It is then clear that two maps $\C/v\to\D$ in $\C\downarrow\dgcat$ corresponding to morphisms $w$ and $w'$
		are homotopic if and only $w$ and $w'$ differ by a boundary in $\Hom(\D)$. 
\end{proof}	
\begin{cor}
	The functor $\kappa_v$  lifts to a functor on the homotopy category of  $\C\downarrow\dgcat\to\operatorname{Sets}$.
\end{cor}
The killing construction has an analogue in the category $\DGA$, which we will now sketch. Let $C$ be a dg algebra and $v\in C_n$ be an $n$-cycle in $A$. Regarding $C$ as a dg category with one object, we can form the killing construction $C/v$. The definition of the functor $\kappa_v$ is likewise similar.
\begin{defi} Given a dg algebra $C$ supplied with a dg map $f:C\to D$, set
	\[
	\kappa_v(C)=\{w\in C_{n+1}:d(w)=f(v)\}/B_{n+1}(D).
	\] 
\end{defi}
As before, it is easy to see that $\kappa_v$ does not depend, up to a natural isomorphism, on the choice of $v$ inside its homology class. The following result holds (which implies, as before, that $\kappa_v$ lifts to the homotopy category of $C\downarrow\DGA$).
\begin{prop}\label{prop:killalg} Given a dg algebra $A$, a cycle $v\in C_n$ and a dg algebra $D$ supplied with a dg map $C\to D$ there is a natural isomorphism  of sets
	\[
	\kappa_v(D)\cong [C/v,D]_{C\downarrow\DGA}
	\]
	where the right hand side above refers to homotopy classes of maps in $C\downarrow\DGA$. In other words, the functor $\kappa_v$ is represented up to homotopy by $C/v$.
\end{prop} 
\begin{proof}
Note that the category $\DGA$ of dg algebras can be viewed as a full subcategory of $\ground_{}\downarrow\dgcat$ where $\ground$  stands for the dg category with one object and $\ground$ worth of morphisms. This inclusion is \emph{not} compatible with model structures on $\DGA$ and $\dgcat$ since not every surjective map between dg algebras is a fibration of the corresponding dg categories. Nevertheless, viewing the homotopy category as a localization of $\DGA$ either on its own or inside $\ground\downarrow\dgcat$, we see that the same maps are being inverted and therefore the above inclusion induces a full embedding on the corresponding homotopy categories. 

With this, the desired result follows directly from Proposition \ref{prop:killcat}.
\end{proof}
\begin{rem}
	It is also possible to carry out the proof of Proposition \ref{prop:killalg} solely inside $\DGA$ where similar arguments would apply.
\end{rem}
\section{One-sided inversion in dg categories}
We will describe the procedure of one-sided (left or right) inversion in a dg category in a homotopy invariant fashion. The arguments follow those in the previous section and are only slightly more complicated (owing to the more complicated nature of the relevant relative cylinder).

Let $\C$ be a dg category, $O_1,O_2$ be objects of $\C$ and $v\in H_n(\Hom(O_1,O_2))$. We will construct a functor
\[ \rho_v:\C\downarrow\dgcat\to\operatorname{Sets}\]
from the undercategory of $\C$ to sets.
\begin{defi} Given a dg category $\D$ supplied with a dg functor $F:\C\to\D$, set
	\[
	\rho_v(\D)=\{w\in H_{-n}\Hom_{\D}(F(O_2),F(O_1)):F(v)F(w)=1\in H_0\End_{\D}(F(O_2))\}
	\] 
\end{defi}
\begin{rem}
The set $\rho_v(\D)$ is either empty or it is noncanonically isomorphic to the set of right zero-divisors of $F(v)$ in $H_n\Hom_{\D}(F(Y),F(X))$ 
(since any two elements in $\rho_v(\D)$ 
differ by a right zero divisor of $F(v)$). Since homotopic functors between dg categories 
induce the same maps on the graded homology categories, it follows that the functor $\rho_v$ lifts to the homotopy category of $\C\downarrow\dgcat$.
\end{rem}
Recall that we denoted by $\ground\langle v\rangle$ the category with two objects $O_1$ and $O_2$ and having its morphisms generated by a single map $v:O_1\to O_2$ between them, with $|v|=n$. Let $Q_v$ stand for the category with the same two objects and generated by three morphisms $v:O_1\to O_2$, $w:O_2\to O_1$ and $u:O_2\to O_2$, with the differential $d(u)=vw-1$ and $d(v)=d(w)=0$. The degrees of $v,w$ and $u$ are $n,-n$ and $1$ respectively. 

Imposing the relation $u=0$ in $Q_v$ we obtain a category that we denote by $\ground\langle v^-\rangle$. The following result shows that $Q_v$ is a resolution of $\ground\langle v^-\rangle$.
\begin{prop}\label{prop: Qresolution}
	The quotient map 
	$Q_v\to \ground\langle v^-\rangle$ is a quasi-equivalence. 
\end{prop}	 
\begin{proof}
	Let $Q$ be the quiver with two vertices $O_1, O_2$ and two arrows $v:O_1\to O_2, w:O_2\to O_1$. Let $Q\langle u\rangle$ be the same quiver with the additional arrow $u:O_2\to O_2$. Denote by $1_{O_1}$ and $1_{O_2}$ idempotents associated with vertices $O_1$ and $O_2$. We consider the algebra $\ground Q/(vw-1_{O_2})$ and the  dg algebra $\ground Q\langle u\rangle$ with $d(u)=vw-1_{O_{2}}$. Note that $\ground Q/(vw-1_{O_2})$ is the category algebra of $\ground\langle v^-\rangle$ and  $\ground Q\langle u\rangle$ with the indicated differential, is the category algebra of $Q_v$. The functor $Q_v\to \ground\langle v^-\rangle$, being bijective on objects, induces the quotient map $j:\ground Q\langle u\rangle\to \ground\langle v^-\rangle$ with $d(u)=vw-1_{O_{2}}$. The desired statement is equivalent to $j$ being a quasi-isomorphism. 
	
Let us consider the modified differential on the complex 
%$(\ground Q_v\to \ground\langle v^-\rangle)$
$\ground Q\langle u \rangle$ given by $d'(u)=vw$. Since the commutative semisimple algebra $\ground\times\ground$ splits off 
%$(\ground Q_v\to \ground\langle v^-\rangle)$,
$\ground Q\langle u \rangle$
 as a $\ground\times\ground$-bimodule (though not as a dg algebra), the homology with respect to the modified differential is unchanged as a graded vector space. But %$(\ground Q_v\to \ground\langle v^-\rangle, d')$
 $(\ground Q\langle u \rangle,d')$
  is the bimodule cobar-construction of the graded Koszul algebra $\ground Q/(vw)$. The homology of this cobar-construction is, therefore, spanned over $\ground\times \ground$ by the elements $v,w$ and $vw$. Therefore the homology of %$(Q_v\to \ground\langle v^-\rangle)$ 
  $kQ\langle u \rangle$
  with the old differential $d$ is spanned over $\ground\times \ground$ by the same elements. The desired claim follows.
\end{proof}	

Clearly, $\ground\langle v\rangle$ is a subcategory of $Q_v$ and $Q_v$ is cofibrant over $\ground \langle v\rangle$. This allows us to form the dg category $R_v(\C):=\C\coprod_{\ground\langle v\rangle}Q_v$. 
\begin{cor}\label{cor:derivedpushout}
	There is a weak equivalence $R_v(\C)\simeq \C\coprod^{\mathbb L}_{\ground\langle v\rangle}\ground[v^-]$.
\end{cor}
\begin{proof}
	This is immediate from Proposition \ref{prop: Qresolution}.
\end{proof}	

It is clear that $R_v(\C)$ is supplied with a dg functor from $\C$, that is to say it is an object in $\C\downarrow\dgcat$.

The following result holds.
\begin{prop}\label{prop:rightsided} Given a dg category $\C$, a homology class $v\in H_n\Hom(\C)$ and a dg category $\D$ supplied with a dg functor $\C\to\D$ there is a natural isomorphism  of sets
	\[
	\rho_v(\D)\cong [R_v(\C),\D]_{\C\downarrow\dgcat}
	\]
where the right hand side above refers to homotopy classes of maps in $\C\downarrow\dgcat$. In other words, the functor $\rho_v$ is represented up to homotopy by $R_v$.
\end{prop}
\begin{proof}
	Let $F:R_v(\C) \to\D$ be a map in $\C\downarrow\dgcat$. As such, it is determined by the images of the morphisms $w$ and $u$ in $\D$. The morphism $F(w)$ is a cycle in $\Hom_{\D}(F(X),F(Y))$ and $d(F(u))=F(v)F(w)-1$.  Therefore, the functor $F$ determines an element in $\rho_v(\D)$ and every element in $\rho_v(\D)$ comes from a functor $R_v\to \D$.
	
Let us suppose that two maps $F_1,F_2:R_v \to\D$ are homotopic in $\C\downarrow\dgcat$; so there is a dg functor $
H:C_{\C}(R_v)\to\D$ implementing a homotopy between $F_1$ and $F_2$. The dg category $
C_{\C}(R_v)$ has the same objects as $R_v$, and it is freely generated over $R_v$ by  morphisms $w':O_2\to O_1$, $u':O_2\to O_2$, $s_1:O_2\to O_1$ and $s_2:O_2\to O_2$. The non-zero differentials on these additional generators are given by the formulas:
\begin{align*}
d(u') & =  vw'-1\\
d(s_1) & =  w-w'\\
d(s_2) & = u-u'+vs_1.
\end{align*}
It follows that the $F_1(w),F_2(w)\in\Hom(\D)$ are homologous cycles since $$F_1(w)-F_2(w')=H(w)-H(w')=dH(s_1).$$ 
So the condition that $F_1$ and $F_2$ are homotopic implies that the homology classes  of right inverses to $F_1(w)$ and $F_2(w)$ coincide. In other words, we have proved that there is a well-defined map 
$[R_v,\D]_{\C\downarrow\dgcat}\to\rho_v(\D)$ and that it is surjective.

To show that it is injective, suppose that we have two maps $F_1$ and $F_2$ as above, such that $F_1(w),F_2(w)\in\Hom(\D)$ are homologous cycles.  We can assume that $F_1(w)$ and $F_2(w)$ are in fact equal; indeed if this is not the case, then replace $F_1$ with a homotopic map $F_1'$ where the corresponding homotopy has $s_2$-component zero and for which $F_1'(w)=F_2(w)$. Then the homology class of $F_1(u)-F_2(u)$ is well-defined.

Let us define a homotopy $H:C_{\C}(R_v)\to \mathcal D$ by setting $H(w)=F_1(w)$, $H(w')=F_2(w')$, $H(u)=F_1(u)$, $H(u')=F_2(u)$,  $H(s_1)=F_1(u)-F_2(u)$ and $H(s_2)=v(F_1(u)-F_2(u))$. Then $H$ is a homotopy between $F_1$ and $F_2$, proving the required injectivity. 
\end{proof}	

There is an analogue of this result in the category $\DGA$. Let $C$ be a dg algebra and $v\in H_n(C)$. Regarding $C$ as a dg category, we can form the corresponding one-sided derived localization $R_v(C)$. The definition of $\rho_v$ is likewise similar.

\begin{defi} Given a dg algebra $D$ supplied with a dg map $F:\C\to\D$, set
	\[
	\rho_v(D)=\{w\in H_{-n}(D):F(v)F(w)=1\in H_0(D).\}
	\] 
\end{defi}
Then the following result holds, which is a direct corollary of Proposition \ref{prop:rightsided}.
\begin{prop}\label{prop:rightsidedalg} Given a dg algebra $C$, a homology class $v\in H_n(C)$ and a dg algebra $D$ supplied with a dg map $C\to D$, there is a natural isomorphism  of sets
	\[
	\rho_v(D)\cong [R_v,\D]_{C\downarrow\DGA}
	\]
	where the right hand side above refers to homotopy classes of maps in $C\downarrow\DGA$. In other words, the functor $\rho_v$ is represented up to homotopy by $R_v$.
\end{prop}
\begin{rem}
The construction of one-sided localization $R_v$ can be carried out entirely in the category $\DGA$. The analogue of $Q_v$ is played by the algebra $Q_v^{\text {alg}}$ which is free on generators $v, w$ and $u$ and having the differential $d(u)=vw-1$. As in the categorical case, it is
 a resolution of the algebra $\ground[v^-]^{\text {alg}}:=\ground\langle v, w\rangle/(vw=1)$. Therefore, for an dg algebra $A$ and its cycle $v$, we have 
 \[R_v(A)\cong A\coprod _{\ground[v]}Q_v^{\text{alg}}\simeq A\coprod^{\mathbb L} _{\ground[v]}\ground[v^-]^{\text {alg}}\]
 We omit the details.
\end{rem} 
\section{Two-sided inverses from one-sided ones}
In this section we apply one-sided localization in dg categories to obtain the ordinary (i.e. two-sided one). This recovers some of the results of \cite{Toen} and \cite{BCL}.

Recall that given a dg category $\C$ and a morphism $v\in \Hom_{\C}(v,w)$, we have defined the one-sided (more precisely, right-sided) localization $R_v(\C)$ as 
\[
R_v(\C):=\C\coprod_{\ground\langle v\rangle} Q_v\simeq \C\coprod^{\mathbb L}_{\ground\langle v\rangle}\ground\langle v^{-}\rangle,
\]
where $\ground\langle v\rangle$ is the category with two objects and a one-dimensional space of morphisms between them spanned by $v$, $\ground\langle v^{-}\rangle$ is the same category where $v$ has been freely right-inverted, and $Q_v$ is a cofibrant replacement of $\ground\langle v^{-}\rangle$.

Symmetrically, let us define the category $\ground \langle{}^{-}v\rangle$ generated over 
$\ground\langle v\rangle$ by a morphism $w'$ such that $w'v=1$. Arguing as 
in Proposition \ref{prop: Qresolution}, we construct a resolution 
$L_v$ of $\ground \langle{}^{-}v\rangle$ 
with generators $v, w', u'$ and the differential $d(u')=w'v-1$. This allows us to define the (derived) left localization $L_v(\C)$ of $\C$ at $v$ as
\[
L_v(\C):=\C\coprod_{\ground\langle v\rangle} L_v\simeq \C\coprod^{\mathbb L}_{\ground\langle v\rangle}\ground\langle {^-}v\rangle.
\]
We will construct a functor 
\[ \lambda_v:\C\downarrow\dgcat\to\operatorname{Sets}\]
from the undercategory of $\C$ to sets; it is a left-sided version of the functor $\rho_v$ considered earlier.
\begin{defi} Given a dg category $\D$ supplied with a dg functor $F:\C\to\D$, set
	\[
	\lambda_v(\D)=\{w\in H_{-n}\Hom_{\D}(F(O_2),F(O_1)):F(w)F(v)=1\in H_0\End_{\D}(F(O_1))\}
	\] 
\end{defi}
Then we have the following result whose proof is similar to that of Proposition \ref{prop:leftsided}.
\begin{prop}\label{prop:leftsided} Given a dg category $\C$, a homology class $v\in H_n\Hom(\C)$ and a dg category $\D$ supplied with a dg functor $\C\to\D$ there is a natural isomorphism  of sets
	\[
	\lambda_v(\D)\cong [L_v(\C),\D]_{\C\downarrow\dgcat}
	\]
	where the right hand side above refers to homotopy classes of maps in $\C\downarrow\dgcat$. In other words, the functor $\lambda_v$ is represented up to homotopy by $L_v$.
\end{prop}
\noproof

Let us now introduce the dg category $I_v$ as follows: \[I_v:=L_v(R_v(\ground\langle v\rangle))\cong R_v(L_v(\ground\langle v\rangle)).\]
It is clear that $I_v$ is cofibrant over $\ground\langle v\rangle$; it has the same objects $O_1,O_2$ as $\ground\langle v\rangle$ and is freely generated by the morphisms $v:O_1\to O_2$, $w, w':O_2\to O_1$, $u:O_2\to O_2$ and $u':O_1\to O_1$. The differentials are defined thus: $d(u)=vw-1, d(u')=w'v-1$.

Let us now denote by $\ground \langle v, v^{-1}\rangle$ the category $\ground \langle v \rangle$ with $v$ inverted on both sides; in other words $\Hom(O_1,O_2)$ and $\Hom(O_2,O_1)$ are one-dimensional vector spaces spanned by $v, v^{-1}$ with $v\circ v^{-1}=\id_{O_1}; v^{-1}\circ v=\id_{O_2}$. Then we have the following result.

\begin{prop}\label{prop:2sidedresolution}
	The functor $I_v\to \ground \langle v, v^{-1}\rangle$ given by setting $u,u'$ to zero, is a quasi-equivalence; in other words, $I_v$ is a cofibrant replacement of $\ground \langle v, v^{-1}\rangle$. 
\end{prop}
\begin{proof}
The claimed result can be proved algebraically, similar to \ref{prop: Qresolution}; we give a more geometric proof. Consider the simplicial set $K$ with two vertices (corresponing to the objects $O_1$ and $O_2$), three non-degenerate 1-simplices connecting them (corresponding to the morphisms $v, w$ and $w'$) and two further non-degenerate 1-simplices	(corresponding to the morphisms $u$ and $u'$). The geometric realization of $K$ is a 2-disc; the two vertices are two diametrically opposite points on its boundary, the three 1-cells are the diameter and two half-circles connecting them, and the remaining two 2-cells are the two half-discs bounded by the two half-circles and the given diameter. The category $I_v$ is obtained from $K$ by the application of the categorical cobar-construction to $C_*(K)$, the simplicial chain coalgebra of $K$ (cf. \cite[Section 3]{HL} regarding this notion; note also, that $I_v$ is the value of the left adjoint to the dg nerve functor, cf. \cite[Section 1.3.1]{Lurie}).

Since the geometric realization of $K$ is a 2-disc, it is contractible. By direct inspection $K$ is group-like, i.e. its fundamental category is a groupoid; therefore by \cite[Corollary 4.20]{HL}, the dg category $I_v$ is quasi-equivalent to the category with a single object and $\ground$ worth of its endomorphisms. Therefore, the quotient map $I\to \ground \langle v, v^{-1}\rangle$ is a quasi-equivalence as claimed.
\end{proof}	
\begin{rem}
	The simplicial set $K$ is a (two-sided) localization of the simplicial interval viewed as an $\infty$-category; this point of view is presented in \cite[Section 3.3]{Cisinski}.
\end{rem}
We can now define the derived (two-sided) localization of a dg category at a given morphism.
\begin{defi}\label{def:2sidedlocalization}
Let $\C$ be a dg category, $O_1, O_2$ be two objects of $\C$ and $v\in H_n\Hom(O_1,O_2)$ determining a homotopy class of a functor $\ground\langle v\rangle\to \C$. Then the \emph{derived localization} of $\C$ at $v$ is the dg category \[\mathbb{L}_v\C:=\C\coprod_{\ground\langle v\rangle}I_v\simeq \C\coprod^{\mathbb{L}}_{\ground\langle v\rangle}\ground \langle v, v^{-1}\rangle.\]
\end{defi}

The dg category $\mathbb{L}_v\C$ comes equipped with (a homotopy class of) a functor $F:\C\to \mathbb{L}_v\C$; moreover $F(v)$ is an invertible morphism in $H_n(\mathbb{L}_v\C)$. It turns out that $\mathbb{L}_v\C$ is universal for this property.

\begin{prop}\label{prop:2sideduniversal}
	Let $G:\C\to \D$ be a dg functor for which $G(v)$ is invertible in $H_*(\D)$. Then there exists a unique up to homotopy functor $\mathbb{L}_v\C\to \D$ making the following diagram commute in the homotopy category of $\C\downarrow\dgcat$:
\[	\xymatrix{
	\C\ar^{F}[r]\ar^G[d]&\mathbb{L}_v\C\ar@{-->}[dl]\\
	\D	
	}
\]
\end{prop}	
\begin{proof}
	Since the inverse to $G(v)$ in $H_*(\D)$ is unique, Proposition \ref{prop:rightsided} implies that $G$ extends to a functor $R_v(\C)\to\D$ uniquely up to homotopy in the undercategory of $\C$. Similarly, Proposition \ref{prop:leftsided} implies that it further extends to a functor $L_vR_v(\C)\simeq \mathbb{L}_v(\C)\to\D$ uniquely up to homotopy in the undercategory of $R_v(\C)$ and a fortiori, uniquely in the homotopy undercategory of $\C$.
\end{proof}

%\begin{rem} Let $\C$ be a dg category, and let $\B$ be the full subcategory on an object $B$ of $\C$. The Drinfield's quotient dg category $\C\to \C/\B$ is isomorphic to $C\to \mathbb{L}_{1_B}\C$ the homotopy category of $\C\downarrow\dgcat$.
%\end{rem}

Now, let $C$ be a dg algebra and $v\in H_n(A)$. Then $\mathbb{L}_v(C)$ is a dg algebra supplied with a map $f:C\to \mathbb{L}_v(C)$. The following result holds.
	\begin{prop}\label{prop:2sideduniversalalg}
		Let $g:C\to D$ be a dg algebra map for which $v(v)$ is invertible in $H_*(D)$. Then there exists a unique up to homotopy map $\mathbb{L}_v C\to D$ making the following diagram commute in the homotopy category o $C\downarrow\DGA$:
		\[	\xymatrix{
			C\ar^{f}[r]\ar^g[d]&\mathbb{L}_v\C\ar@{-->}[dl]\\
			D	
		}
		\]
	\end{prop}
\begin{proof}
	Completely analogous to the proof of Proposition \ref{prop:2sideduniversal} taking into account Proposition \ref{prop:rightsidedalg} and its left-sided analogue.
		\end{proof}

\section{One-sided localization in pretriangulated categories}
In this section we show how killing morphisms in pretriangulated dg categories is related to one-sided inversion of morphisms. 
 \begin{theorem}\label{killingvslocalisation}
Let $\C$ be a pretriangulated dg category, and let
\begin{equation}\label{triangleinpretriangulated}
	\begin{tikzcd}
		A \arrow[r,"x"] &
		B \arrow[r,"v"] &
		C \arrow[r,"s"] & \Sigma A
	\end{tikzcd}
\end{equation}
 be a triangle in $\mathcal{C}$. 
Then we have weak equivalences 
 	$$ \mathcal{C}/v \simeq R_x(\mathcal{C})$$
 	and
 		$$ \mathcal{C}/v \simeq L_s(\mathcal{C})$$
  in $\C\downarrow\dgcat$.
  \end{theorem}

\begin{proof}
	We prove the first weak equivalence; the argument for the second is similar.
 The morphisms in the triangle
 (\ref{triangleinpretriangulated})
 are degree $0$ cycles, and the composition of any two is a boundary. So we have $vx=d(t)$ and $(\Sigma x)s=d(u)$ for some morphisms $t$ and $u$ of degree $1$; it is always possible to choose $t$ and $u$ so that $vu+ts$ is homologous to $1_C$.

For a dg category $\mathcal{D}$ over $\mathcal{C}$, we wish to compare
$$\rho_x(\mathcal{D}) = \left\{ y\in Z_0\Hom_{\mathcal{D}}(B,A)\mid xy-1_B\in B_0\Hom_{\mathcal{D}}(B,B)\right\}/B_0\Hom_{\mathcal{D}}(B,A)$$
and
$$\kappa_v(\mathcal{D})=\left\{w\in\Hom_{\mathcal{D}}(B,C)_1\mid d(w)=v\right\}/B_1\Hom_{\mathcal{D}}(B,C).$$

Given $y\in \rho_x(\mathcal{D})$, choose its representative $y'\in Z_0\Hom_{\mathcal{D}}(B,A)$ and $z\in \Hom_{\mathcal{D}}(B,B)_1$ such that
$d(z)=1-xy'$. Then $d(vz)=v(1-xy')=v-d(ty')$ so $$w_z:=vz+ty' \mod B_1\Hom_{\mathcal{D}}(B,C)\in \kappa_v(\mathcal{D}).$$ If  $z'\in \Hom_{\mathcal{D}}(B,B)_1$ is another element such that
$d(z)=1-xy'$ then 
$w_z-w_{z'}=v(z-z')\in B_1\Hom(B,C)$ so 
$w_z$ and $w_{z'}$ determine the same class in  $\kappa_v(\mathcal{B})$. Similarly, the class of $w_z$ modulo $B_1\Hom_{\mathcal{D}}(B,C)$ does not depend on the choice of $y'$. 
The correspondence $y\mapsto w_z$ defines a map
\begin{equation}\label{comparisonmap}
	j_{\mathcal{D}}:	\rho_x(\mathcal{D}) \longrightarrow \kappa_v(\mathcal{D}),
	\end{equation}
which defines a morphism of functors $j:\rho_x \to \kappa_v$. On representing objects, it is given by the morphism $$J: \mathcal{C}/v \to R_x(\mathcal{C})$$
under $\mathcal{C}$
that sends $w$ to $vz+ty$. Here $w,y,z$ are the elements arising in the construction of
 $\mathcal{C}/v $ and $R_x(\mathcal{C})$ from $\mathcal{C}$, and $t$ is the morphism in $\mathcal{C}$ fixed above.

Assume now that $\mathcal{D}$ is pretriangulated. Considering the image of (\ref{triangleinpretriangulated}) in the triangulated category $H_0(\mathcal{D})$, we see that 
$[v]\in H_0 \Hom_{\mathcal{D}}(B,C)=0$ if and only if 
$[x]\in H_0 \Hom_{\mathcal{D}}(A,B)$ is right invertible. This implies that 
$\rho_x(\mathcal{D})=\emptyset$ if and only if $\kappa_v(\mathcal{D})=\emptyset$.
So assume both are nonempty.
 Applying the functor $\Hom_{\mathcal{D}}(B,-)$ to (\ref{triangleinpretriangulated}) we obtain a short exact sequence (since $v$ is zero in homology by our assumption):
\[
\xymatrix{
	0\ar[r]&H_1\Hom_{\mathcal{D}}(B,C)\ar^{[s]\circ -}[r]&H_0\Hom_{\mathcal{D}}(B,A)\ar^{[x]\circ -}[r]&H_0\Hom_{\mathcal{D}}(B,B)\ar[r]&0
}
\]
The set $\rho_x(\mathcal{B})$ is identified with the preimage of 
$\id\in H_0\Hom(B,B)$ and so, is an affine space for 
$H_1\Hom(B,C)$. The set $\kappa_v(\mathcal{D})$ is likewise an affine space for $H_1\Hom(B,C)$.
 We claim that the map $j_{\mathcal{D}}:\rho_x(\mathcal{D}) \longrightarrow \kappa_v(\mathcal{D})$ constructed above is compatible with this affine space structure (so it is an isomorphism).

To prove the claim, take $\alpha\in H_1\Hom_{\D}(B,C)$; then the action of 
$\alpha$ on $y\in H_0\Hom_{\D}(B,A)$ has the form
$\alpha\cdot y=y+s\alpha$. Furthermore, $x(y+s\alpha)=1+d(z+u\alpha)$ and therefore,
\[
j(\alpha\cdot y)=v(z+u\alpha)+t(y+s\alpha)=w_z+(vu+ts)\alpha  =  w_z + \alpha = \alpha\cdot w_z
\]
since by assumption $vu+ts$ is homologous to $1_C$.

We have shown that $j_{\mathcal{D}}$ is a bijection for pretriangulated dg categories $\mathcal{D}$ under $\mathcal{C}$. We deduce that $J$ induces an isomorphism in the homotopy category of the undercategory of $\mathcal{C}$ in the Morita model of the category of dg categories, i.e. the map induced on pretriangulated hulls by $J$ is a weak equivalence in $\C\downarrow\dgcat$. Now note that $J$ is essentially surjective; in fact, the objects of $\mathcal{C}/v$ and $R_x(\mathcal{C})$ both coincide with $\C$, and $J$ is the identity on objects. Since the Yoneda functor from any dg category into its pretriangulated hull is quasi fully faithful, we deduce that $J$ was already an isomorphism in the homotopy category of $\C\downarrow\dgcat$.
\end{proof}

By rotating the triangle in the theorem above, we obtain the following.
\begin{cor}
	Let $\C$ be a pretriangulated dg category, and let
	\begin{equation*}
		\begin{tikzcd}
			A \arrow[r,"x"] &
			B \arrow[r,"y"] &
			C \rightsquigarrow
		\end{tikzcd}
	\end{equation*}
	be a triangle in $\mathcal{C}$. 
	Then we have a weak equivalence 
	$$  R_y(\mathcal{C})
\simeq L_x(\mathcal{C})$$
	in $\C\downarrow\dgcat$.
	\end{cor}

Finally, we deduce an interpretation of Drinfeld's quotient \cite{Drinfeld} as a factorization into two consecutive morphism-killing constructions. Let $\C$ be a pretriangulated dg category, $B$ be an object in $\C$, and denote by $\B$ the full subcategory of $\C$ on $B$. We consider Drinfeld's dg quotient $\C/\B$; recall that it is characterized by various universal properties \cite{Tabuada}.

\begin{cor}\label{Drinfeldfactorisation}
	Given
				\begin{equation*}
				\begin{tikzcd}
					A \arrow[r,"x"] &
					B \arrow[r,"y"] &
					C \rightsquigarrow
				\end{tikzcd}
			\end{equation*}
			 a triangle in $\mathcal{C}$,
		 we have  weak equivalences
			$$\C/\B \simeq (\C/x)/y \simeq (\C/y)/x.$$
			in $\C\downarrow\dgcat$.
\end{cor} 	 

\begin{proof}
	Denote by $z:C\rightarrow\Sigma A$ the connecting morphism for the triangle. We then have
	$$\C/\B \simeq \mathbb{L}_z(\C) \simeq L_zR_z(\C) \simeq (\C/x)/y,$$
	by Lemma~\ref{Drinfeldquotientaslocalisation}, the connection between two-sided and one-sided localisation established in Section 5 and Theorem~\ref{killingvslocalisation}. 
	\end{proof}
The following result, used in the proof of Corollary \ref{Drinfeldfactorisation}, is unsurprising and well known; we include an argument for the reader's convenience.
\begin{lem}\label{Drinfeldquotientaslocalisation}
		Given
	\begin{equation*}
		\begin{tikzcd}
			A \arrow[r] &
			B \arrow[r] &
			C \arrow[r,"z"] & \Sigma A
		\end{tikzcd}
	\end{equation*}
a triangle in $\C$, we have a weak equivalence
$$\C/\B \simeq \mathbb{L}_z(\C)$$
in $\C\downarrow\dgcat$.
	\end{lem}
\begin{proof}
The dg categories $\C/\B$ and $\mathbb{L}_z(\C)$ have quasi-equivalent pretriangulated hulls $[\C/\B]$ and $[\mathbb{L}_z(\C)]$ since they have identical universal properties according to \cite[Theorem 4.0.1]{Tabuada}. The resulting quasi-equivalence $[\C/\B]\to [\mathbb{L}_z(\C)]$ restricts to a dg functor $\C/\B\to \mathbb{L}_z(\C)$ and is a quasi-equivalence since the Yoneda embeddings $\C/\B\to[ \C/\B]$ and $\mathbb{L}_z(\C)\to [\mathbb{L}_z(\C)]$ are quasi-fully faithful dg functors.
\end{proof}	

\end{document}